\documentclass{article}
\usepackage{fullpage}
\usepackage{amsfonts, amsmath, amssymb,latexsym,epsfig}
\usepackage{amsthm}
\usepackage{tikz}
\usetikzlibrary{graphs}
\usetikzlibrary{graphs.standard}
\usepackage{relsize}
\usepackage{verbatim}
\usepackage{enumerate}

\newtheorem{thm}{Theorem}[section]
\newtheorem{cor}[thm]{Corollary}
\newtheorem{lem}[thm]{Lemma}
\newtheorem{remark}[thm]{Remark}

\newtheorem{conj}[thm]{Conjecture}

\newcommand{\C}{\mathbb C}

\newcommand{\F}{\mathbb F}

\newcommand{\hs}{\hspace{1cm}}

\author{

Himanshu Gupta\footnotemark[1]

and Vladislav Taranchuk\footnotemark[2]
}
\title{On the eigenvalues of the graphs $D(5,q)$}

\providecommand{\keywords}[1]{\textbf{\textit{Keywords}} #1}

\begin{document}

\maketitle
\footnotetext[1]{Department of Mathematical Sciences, University of Delaware, Newark, DE 19716-2553, USA, {\tt himanshu@udel.edu}.}
\footnotetext[2]{Department of Mathematical Sciences, University of Delaware, Newark, DE 19716-2553, USA, {\tt vladtar@udel.edu}.}

\begin{abstract}

    Let $q = p^e$, where $p$ is a prime and $e$ is a positive integer. The family of graphs $D(k, q)$, defined for any positive integer $k$ and prime power $q$, were introduced by Lazebnik and Ustimenko in 1995. To this day, the connected components of the graphs $D(k, q)$, provide the best known general lower bound for the size of a graph of given order and given girth. Furthermore, Ustimenko conjectured that the second largest eigenvalue of $D(k, q)$ is always less than or equal to $2\sqrt{q}$. If true, this would imply that for a fixed $q$ and $k$ growing, $D(k, q)$ would define a family of expanders that are nearly Ramanujan. In this paper we prove the smallest open case of the conjecture, showing that for all odd prime powers $q$, the second largest eigenvalue of $D(5, q)$ is less than or equal to $2\sqrt{q}$.

\end{abstract}

\keywords{
Expanders, Eigenvalues of graphs, Graph spectrum, Algebraically defined graphs, Cayley graphs, Irreducible representations, Character sums.
}

\section{Introduction and Motivation}\label{Section_intro}
Let $\Gamma$ be a graph with vertex set $V(\Gamma)$ and edge set $E(\Gamma)$. We use standard terminology from graph theory, see, e.g., Bollob\'{a}s \cite{B}. We refer to $|V(\Gamma)|$ and $|E(\Gamma)|$ as the order and the size of $\Gamma$, respectively. A graph is called $q$-regular if every vertex has exactly $q$ neighbors. All the graphs we consider are $q$-regular and simple (i.e., undirected, without loops and multiple edges). 

The {\em Cheeger constant} $h(\Gamma)$ is an important measure of the connectivity of graph $\Gamma$. It is defined as
\begin{align*}
h(\Gamma) = \min\left\{\frac{|\partial S|}{|S|}: S \subset V(\Gamma), 0 < |S| \leq \frac{|V(\Gamma)|}{2}\right\}
\end{align*}
where $\partial S$ denotes the set of edges with one endpoint in $S$ and another endpoint in $V(\Gamma)\backslash S$. An infinite family of $q$-regular graphs $\Gamma_n$ is called {\em expander family} if $h(\Gamma_n)$ are uniformly bounded away from $0$ together with $|V(\Gamma_n)| \to \infty$ as $n \to \infty$. We refer an interested reader to \cite{DSV,HLW,KS} for examples of families of expanders, their theory, history, and applications. 

The adjacency matrix $A$ of a graph $\Gamma$ has its rows and columns indexed by $V(\Gamma)$ and $A_{x,y} = 1$ if there is an edge between $x$ and $y$ and 0 otherwise. Since $A$ is a real symmetric matrix, all of its eigenvalues are real. We then call the sequence $q= \lambda_1 \geq \lambda_2 \geq \hdots \geq \lambda_{|V(\Gamma)|}$ of the eigenvalues of $A$ the {\em spectrum} of $\Gamma$. In general, it is hard to compute $h(\Gamma)$. However, there is an interesting connection between $h(\Gamma)$ and the spectrum of $\Gamma$, given by the following general result. 
\begin{thm}[Alon-Milman \cite{AM,A}, Dodziuk \cite{D}, and Mohar \cite{mohar}]
Let $\Gamma$ be a $q$-regular connected graph with second largest eigenvalue $\lambda_2$. Then
\begin{align*}
\frac{(q-\lambda_2)}{2} \leq h(\Gamma) \leq \sqrt{q^2-\lambda_2^2}.
\end{align*}
\end{thm}
See Davidoff, Sarnak, and Valette \cite[Th.\ 1.2.3]{DSV} for a proof. The difference $q-\lambda_2$ of a $q$-regular graph $\Gamma$ is called the {\em spectral gap} of $\Gamma$. Thus, one way of proving a family of regular graphs are expanders is by studying the spectral gap of the graphs in the family. The larger the spectral gap, the better the expansion. However, there is a limit to how large the spectral gap can be due to the famous Alon-Boppana theorem.

\begin{thm}[Alon-Boppana \cite{A,N}]
Let $\{\Gamma_n\}$ be an infinite family of $q$-regular connected graphs with $|V(\Gamma_n)| \rightarrow \infty$ as $n \rightarrow \infty$. Then  
$$
{\liminf_{n\to \infty}} \lambda_2(\Gamma_n) \geq 2\sqrt{q-1}.
$$
\end{thm}
A $q$-regular connected graph $\Gamma$ is called {\em Ramanujan} if $\lambda_2(\Gamma) \leq 2\sqrt{q-1}$. Thus, a family of $q$-regular Ramanujan graphs is an asymptotically best possible expander family with respect to the spectral lower bound. One of the central problems in spectral graph theory is to give explicit constructions for families of expanders. Even better, but more difficult, is to give constructions of infinite families of Ramanujan graphs. Let $G$ be a group, and let $S$ be a nonempty finite subset of $G$ such that $S$ is closed under inverses, i.e., $S = S^{-1}$. The {\em Cayley graph} $\text{Cay}(G,S)$, is the graph with vertex set $G$, and two vertices $x$ and $y$ are adjacent if and only if $y = xs$ for some $s \in S$. The celebrated constructions of Ramanujan graphs by Lubotzky-Phillips-Sarnak \cite{LPS}, and independently by Margulis \cite{M} are Cayley graphs. These constructions are extended by Morgenstern in \cite{MOR}, and are thoroughly explained in \cite{DSV}. 

 In this paper, we are interested in a conjecture of Ustimenko. Informally speaking, the conjecture claims that the families of $q$-regular bipartite graphs constructed by Lazebnik and Ustimenko \cite{LU} are `close' to being Ramanujan for each prime power $q$. 

For rest of the paper, let $q = p^e$, where $p$ is a prime and $e$ is a positive integer. Let $\mathbb{F}_q$ be the finite field of $q$ elements, and $\F_q^k$ be a vector space of dimension $k$ over $\F_q$. The family of $q$-regular bipartite graphs with $2q^k$ vertices, denoted by $D(k,q)$ where $k\geq 2$, was constructed by Lazebnik and Ustimenko in \cite{LU}. The graphs $D(k,q)$ are connected for $k\leq 5$ and $q$ odd, and are disconnected for $k\geq 6$ and $q$ odd; see Lazebnik, Ustimenko, and Woldar \cite{LUW1995, LUW}. As the graphs $D(k, q)$ are edge-transitive, all connected components are isomorphic. Let $CD(k, q)$ denote one of them.
The graphs $CD(k, q)$ provide the best known general lower bound for the size of a graph with given order and given girth. For surveys on these graphs and their applications, see Lazebnik and Woldar \cite{LW}, Lazebnik, Sun and Wang \cite{LSW}, and Sun's Ph.D. thesis \cite[Chapter 2]{Sun}. 

The structural properties of $CD(k,q)$ are well studied. However, the spectrum of $D(k,q)$ is only known for $k=2,3$, and $4$. The spectrum of $D(2,q)$ and $D(3,q)$ was computed by Li, Lu, and Wang in \cite[Sec.\ 5]{LLW}. In fact, graphs in these two families are isomorphic to the particular cases of the Wenger graphs. The spectrum of Wenger graphs were completely determined in Cioab\u{a}, Lazebnik, and Wang \cite{CLL}. The distinct eigenvalues of $D(2,q)$ and $D(3,q)$ are $\{\pm\sqrt{q},0,\pm q\}$ and $\{\pm\sqrt{q},\pm \sqrt{2q},0,\pm q\}$, respectively. For details  of the history of the following conjecture, see Moorehouse, Sun, and Williford \cite{MSW}. 

\begin{conj}[Ustimenko] \label{Main_Conjecture}
For all $(k,q)$, $D(k,q)$ (and therefore $CD(k, q)$) has second largest eigenvalue less than or equal to $2\sqrt{q}$.
\end{conj}

For a graph $\Gamma$ the {\em distance-two graph} $\Gamma^{(2)}$ is the graph with the vertex set as $V(\Gamma)$ and two vertices are adjacent if and only if their distance in $\Gamma$ is exactly two. If $\Gamma$ is a connected bipartite graph, then $\Gamma^{(2)}$ is a graph with two components. Let $V_1$ and $V_2$ denote the partite sets of $\Gamma$. The induced subgraphs on these components are called the {\em halved graphs} of $\Gamma$ denoted by $\Gamma^{(2)}(V_1)$ and $\Gamma^{(2)}(V_2)$.  If $\Gamma$ is a $q$-regular bipartite graph and contains no $4$-cycles, then the halved graphs are $q(q-1)$-regular simple graphs. In this case, there is a simple relationship between the eigenvalues of $\Gamma$ and the eigenvalues of $\Gamma^{(2)}(V_j)$: every eigenvalue $\lambda$ of $\Gamma^{(2)}(V_j)$ with multiplicity $m$ corresponds to a pair of eigenvalues $\pm \sqrt{\lambda+q}$ of $\Gamma$, each with multiplicity $m$ (or a single eigenvalue $0$ of multiplicity $2m$ in the case $\lambda = -q$), see \cite{CLL}. Furthermore, the spectrums of both halved graphs are the same. Consequently, to prove Conjecture \ref{Main_Conjecture} it is enough to prove that the second largest eigenvalue of one of the halved graphs of $CD(k,q)$ is less than or equal to $3q$. 

In 2017, the spectra of the halved graphs of $D(4, q)$ were determined in \cite{MSW} and consequently proving the conjecture for $D(4,q)$. Motivated by their results and ideas, in this paper, we prove the next open case of the conjecture for graphs $CD(5,q)$ whenever $q$ is an odd prime power. In particular, we prove the following theorem.

\begin{thm}\label{Main_Theorem}
Let $q = p^e$, where $p$ is an odd prime and $e$ is a positive integer. Then the second largest eigenvalue of $CD(5,q)$ is less than or equal to $2\sqrt{q}$. 
\end{thm}

For each fixed $q$, the graph $D(6,q)$ is isomorphic to $q$ copies of $D(5,q)$, see \cite{LU, LW}. Hence, the following corollary immediately follows.   

\begin{cor}\label{Main_Corollary}
Let $q = p^e$, where $p$ is an odd prime and $e$ is a positive integer. Then the second largest eigenvalue of $CD(6,q)$ is less than or equal to $2\sqrt{q}$. 
\end{cor}

Our approach to prove Theorem \ref{Main_Theorem} is along the same line as that of \cite{CLL,MSW}. In particular, we find a group $G$ and generating set $S \subset G$ such that one of the halved graphs of $D(5, q)$, is isomorphic to Cay$(G, S)$. It is well known that the representation theory of groups is an important tool for studying the spectra of Cayley graphs, see \cite{Babai}.
With this in mind, we find all the irreducible representations of $G$, which reduces the problem of finding the eigenvalues of the point graph to finding eigenvalues of particular $q \times q$ matrices obtained from sums of irreducible representations of $G$. We determine the eigenvalues of these matrices explicitly as character sums. Using the well established theory of characters over finite fields and Weil's bound for character sums, we are able to prove our theorem. Similar ideas have been used in Cao, et. al. \cite{CLWW} and Cioab\u{a}, Lazebnik, and Sun \cite{CLS} but for different families of graphs. In \cite{CLWW,CLL} the group turned out to be abelian, while in \cite{MSW} it was non-abelian for odd $q$. Note that, if $G$ is an abelian group, then any Cayley graph of $G$ has a diagonlizable adjacency matrix, making the problem of finding the spectrum somewhat simpler. When $G$ is non-abelian, the problem can become considerably more difficult depending on the irreducible representations of $G$. 

The rest of the paper is organized as follows. In Section \ref{Section_point_graph} we give an explicit description of a group $G$ and generating set $S$ such that one of the halved graphs of $D(5,q)$ is isomorphic to $\text{Cay}(G,S)$. In Section \ref{Section_background} we collect necessary results about finite fields and characters over finite fields to use them. In Section \ref{Section_irr_rep} we compute all the irreducible representations of the group $G$, and then describe the $q\times q$ matrices whose spectrum we need to compute. Finally, in Sections \ref{MMatrix} and \ref{NMatrix} we compute their spectrum. Moreover, we prove that the eigenvalues are upper bounded by $3q$ which in turn completes the proof of Theorem \ref{Main_Theorem}.

\section{The graphs \boldmath{$D(5, q)$} and their point collinearity graphs}\label{Section_point_graph}
Let $q = p^e$, where $p$ is a prime and $e$ a positive integer. Define a graph $D(5, q)$ to be a bipartite graph where each part has $q^5$ vertices and represented by a copy of $5$-dimensional vector space $\F_q^5$. Vertices in one part are called as \textit{lines}, and the other part as \textit{points}. We write $[\ell] = [\ell_1, \ell_2, \ell_3, \ell_4, \ell_5]$ to denote a vertex that is a line, and $(p) = (p_1, p_2, p_3, p_4, p_5)$ to denote a vertex that is a point. The adjacency relation between a point $(p)$ and a line $[\ell]$ is given by
\begin{align*}
    p_2 + \ell_2 = p_1\ell_1, \hspace{1cm}
    p_3 + \ell_3 = p_1\ell_2, \hspace{1cm} 
    p_4 + \ell_4 = p_2\ell_1, \hspace{1cm}
    p_5 + \ell_5 = p_3\ell_1.
\end{align*}

\begin{lem}
Let $q = p^e$, where $p$ is an odd prime and $e$ is a positive integer. Further, let $\Gamma(q)$ be a bipartite graph with the same parts as that of $D(5, q)$ and adjacency relation defined by 
\begin{align*}
    p_2 + \ell_2 = p_1\ell_1, \hspace{1cm}
    p_3 + \ell_3 = p_1\ell_1^2, \hspace{1cm}
    p_4 + \ell_4 = p_1^2\ell_1, \hspace{1cm}
    p_5 + \ell_5 = p_1^2\ell_1^2.
\end{align*}
Then $\Gamma(q)$ is isomorphic to $D(5, q)$.
\end{lem}

\begin{proof}
Define the map $\pi:D(5, q) \rightarrow \Gamma(q)$ by
\begin{align*}
    \pi((p_1, p_2, p_3, p_4, p_5)) &= (p_1', p_2', p_3', p_4', p_5') = (p_1, p_2, p_4, p_3 + p_1p_2, 2p_5 + p_2^2) \\
    \pi([\ell_1, \ell_2, \ell_3, \ell_4, \ell_5]) &= [\ell_1', \ell_2', \ell_3', \ell_4', \ell_5'] = [\ell_1, \ell_2, \ell_4 + \ell_1\ell_2, \ell_3, 2\ell_5 + 2\ell_1\ell_3 - \ell_2^2].
\end{align*}

To see that $\pi$ is a bijection, it is enough to show that it is a surjection. Let $(p) = (p_1, p_2, p_3, p_4, p_5)$ be in the image of $\pi$. Then observe that $\pi((p_1, p_2, p_4 - p_1p_2, p_3, 2^{-1}p_5 - p_2^2)) = (p)$ by definition. Likewise, for any $[\ell] = [\ell_1, \ell_2, \ell_3, \ell_4, \ell_5]$, we have that $\pi([\ell_1, \ell_2, \ell_4 - \ell_1\ell_2, \ell_3, 2^{-1}\ell_5 - 2\ell_1(\ell_4 - \ell_1\ell_2) + \ell_2^2]) = [\ell]$. We now confirm that $\pi$ is an isomorphism by verifying that the adjacency relations are preserved:
\begin{align*}
     p_2' + \ell_2' &= p_2 + \ell_2 = p_1\ell_1 = p_1'\ell_1' \\
     p_3' + \ell_3' &= p_4 + \ell_4 + \ell_1\ell_2 = p_2\ell_1 + \ell_1\ell_2 = (p_1\ell_1 - \ell_2)\ell_1 + \ell_1\ell_2 = p_1'(\ell_1')^2 \\ 
     p_4' + \ell_4' &= p_3 + p_1p_2 + \ell_3 = p_1\ell_2 + p_1p_2 = p_1(p_1\ell_1 - p_2) + p_1p_2 = (p_1')^2\ell_1' \\
     p_5' + \ell_5' &= 2(p_5 + \ell_5) + p_2^2  + 2\ell_1\ell_3 - \ell_2^2 
     = 2(p_3\ell_1) + p_2^2 + 2\ell_1\ell_3 - \ell_2^2 \\
     &= 2\ell_1(p_3 + \ell_3) + p_2^2 - \ell_2^2 = 2\ell_1(p_4' - p_1p_2 + \ell_4') + p_2^2 - \ell_2^2   \\
     &= 2\ell_1'((p_1')^2\ell_1' - p_1p_2) + p_2^2 - \ell_2^2  = 2(p_1'\ell_1')^2 - 2p_1'\ell_1'p_2' + (p_2')^2 - (\ell_2')^2 \\
     &= 2(p_1'\ell_1')^2 - 2(p_2' + \ell_2')p_2' + (p_2')^2 - (\ell_2')^2 = 2(p_1'\ell_1')^2 - (p_2' + \ell_2')^2 = (p_1'\ell_1')^2
\end{align*}
Thus $D(5, q)$ is isomorphic to $\Gamma(q)$.
\end{proof}

As a result of the lemma, we work with $\Gamma(q)$ instead of $D(5, q)$. Let $\Gamma^{(2)}(q)$ be the halved graph on the set of points. In other words, the vertex set of $\Gamma^{(2)}(q)$ is the set of points of $\Gamma(q)$ and two vertices $(p)$ and $(r)$ are adjacent in $\Gamma^{(2)}(q)$ if and only if there exists an $[\ell]$ in $\Gamma(q)$ such that $[\ell]$ is adjacent to both $(p)$ and $(r)$ in $\Gamma(q)$. Using the adjacency relations of $\Gamma(q)$, we see that $(p)$ is adjacent to $(r)$ in $\Gamma^{(2)}(q)$ if and only if there exists an $\ell_1 \in \F_q$ such that following system of equations is satisfied
\begin{align}\label{GraphEqs}
    p_2 - r_2 = (p_1-r_1)\ell_1, \hspace{0.1cm}  
    p_3 - r_3 = (p_1 - r_1)\ell_1^2, \hspace{0.1cm} 
    p_4 - r_4 = (p_1^2 - r_1^2)\ell_1, \hspace{0.1cm} 
    p_5 - r_5 = (p_1^2 - r_1^2)\ell_1^2
\end{align}

We call the graph $\Gamma^{(2)}(q)$, the \textit{point collinearity graph} of $\Gamma(q)$, or \textit{point graph} for short. It is easy to see from (\ref{GraphEqs}) that the condition that $(p)$ is distinct from $(r)$ is equivalent to $p_1 \neq r_1$. Furthermore, it is known that the girth of $D(5, q)$ is 10. So $\Gamma^{(2)}(q)$ is a simple graph with no loops.

We want to show that the graph $\Gamma^{(2)}(q)$ is actually a Cayley graph, and for that we need a group. So we define a group $G = (\mathbb{F}_q^5,\cdot)$ with the binary operation as follows:
\begin{align} \label{GroupOp}
    X\cdot Y = (x_1+y_1,x_2+y_2,x_3+y_3,x_4+y_4+2x_1y_2,x_5+y_5+2x_1y_3),
\end{align}
where $X= (x_1,x_2,x_3,x_4,x_5)$ and $Y = (y_1,y_2,y_3,y_4,y_5)$. Note that $G$ is a non-abelian group with $
X^{-1} = (-x_1, -x_2, -x_3, -x_4 + 2x_1x_2, -x_5 +2x_1x_3).$ Later on we will use $XY$ instead of $X\cdot Y$ whenever it is clear from the context. 

\begin{lem}
Let $G$ be defined as above and $
S = \{ (x, xa, xa^2, x^2a, x^2a^2) : a, x \in \F_q, x\neq 0\}
$. Then $S = S^{-1}$ and $\Gamma^{(2)}(q)$ is an undirected Cayley graph of $G$ with generating set $S$, implying $\Gamma^{(2)}(q) = \text{Cay}(G, S)$.
\end{lem}
\begin{proof} If $X = (x,xa,xa^2,x^2a,x^2a^2)$, then $X^{-1} = (-x,-xa,-xa^2,x^2a,x^2a^2)$. Hence,
the first assertion follows immediately. For the second, fix a vertex $(r)$ in $\Gamma^{(2)}(q)$ and consider any point $(p)$ adjacent to $(r)$. Let $p_1 - r_1 = x$ and $\ell_1 = a$. Since any two adjacent vertices in $\Gamma^{(2)}(q)$ must have distinct first coordinate, then $x \neq 0$. We may rewrite the equations in (\ref{GraphEqs}) to say that $(r)$ is adjacent to $(p)$ if and only if there exists an $a, x \in \F_q$, with $x \neq 0$, such that 
\begin{align*}\label{PointGraphEqs}
     &p_2  = r_2 + xa  \nonumber \\
    &p_3  = r_3 + xa^2 \nonumber  \\
    &p_4  = r_4 + x(x + 2r_1)a  = r_4 + x^2a +2xar_1  \\
    &p_5  = r_ 5 + x(x+2r_1)a^2 = r_5 + x^2a^2 + 2xa^2r_1. \qedhere 
\end{align*} 
\end{proof}
\noindent We now collect some results about finite fields and characters which will be needed in proving Theorem \ref{Main_Theorem}.

\section{Background on finite fields and characters}\label{Section_background}

Let $\F_q$ be a finite field, where $q = p^e$, and $\zeta = e^{2\pi i/p}$. The trace map is defined as
$$
tr: \F_q \rightarrow \F_p, \hs tr(a) = a + a^p + \cdots + a^{p^{e-1}}.
$$

\begin{lem} \cite[Ch.\ 5]{LN}
For every polynomial of the form $f(x) = bx+c \in \mathbb{F}_q[x]$ we have
\begin{align*}
\sum_{a \in \mathbb{F}_q} \zeta^{tr(f(a))} = \begin{cases}
0, &\text{if } b\neq 0;\\
q\zeta^{tr(c)}, &\text{otherwise}. 
\end{cases} 
\end{align*} 
\end{lem} 
Let $a \in \F_q^*$ be a primitive element, and $\zeta_{q-1} = e^{2 \pi i/(q-1)}$. Let $k \in \{ 0, 1, \dots, q-1  \}$. The multiplicative character, or just character, $\chi_k : \F_q \rightarrow \C$ is defined by
$$
\chi_k(a^j) = \zeta_{q-1}^{kj}
$$
and $\chi_k(0) = 0$, except when $k = 0$, in which case $\chi_0(0) = 1$.
We denote by $\widehat{\F_q^*}$ the group of all multiplicative characters on $\F_q^*$. The quadratic character $\chi_{(q-1)/2}$ will appear often, and so we will denote it by $\eta$. 

For any $f \in \F_q[y]$, we denote by $S_{f}$ the following exponential sum
$$
S_{f} = \sum_{a \in \F_q} \zeta^{tr(f(a))}.
$$
The following lemma can be found in many places, see, e.g., \cite[Th.\ 5.15 and Th.\ 5.33]{LN}.
\begin{lem}\label{square}
Let $q = p^e$ be an odd prime power and $\zeta = e^{2 \pi i/ p}$. If $g \in \F_q^*$ and $f(y) = gy^2$, then
$$
S_{f} = \sum_{a \in \F_q}\zeta^{tr(ga^2)} = \left\{ \begin{array}{ll}
  \eta(g)(-1)^{e-1} \sqrt{q},  & p \equiv 1 \text{(mod 4)} \\
   \eta(g)(-i)^{e+2} \sqrt{q}, & p \equiv 3 \text{(mod 4)}
\end{array}  \right\} = \eta(g)S_{y^2}.
$$
\end{lem}

\begin{lem} \label{LemSq}
Let $q = p^e$ be an odd prime power. Then
\[
\sum_{a \in \F_q^*}\eta(a^2 - 1) = \left\{\begin{array}{ll}
  -2,   &\mbox{if $q \equiv 1$(mod 4);} \\
  0,  &\mbox{if $q \equiv 3$(mod 4).}
\end{array}  \right.
\]
\end{lem}

\begin{proof}
By \cite[Th.\ 5.48]{LN}, we have
\begin{align}\label{sum}
\sum_{a \in \F_q}\eta(a^2 - 1) = -1.
\end{align}
If we exclude $a = 0$, then we subtract $\eta(-1)$ from both sides of (\ref{sum}). Now, $\eta(-1) = -1$ when $q \equiv 3$(mod 4) and $\eta(-1) = 1$ if $q \equiv 1$(mod 4).
\end{proof}

For the proofs of the following results, see \cite{Schmidt}.

\begin{lem}\label{Schlem}\cite[Ch.\ 1, Lem. 2C]{Schmidt}
Let $d$ be be a positive integer.
Suppose that $y^d - f(x) \in \F_q[x, y]$. Then the following are equivalent:
\begin{enumerate}
    \item $y^d - f(x)$ is absolutely irreducible (i.e., irreducible over the algebraic closure of $\F_q$);
    \item $y^d - cf(x)$ is absolutely irreducible for each $c \in \F_q^*$;
    \item
    if $f(x) = a(x-\alpha_1)^{d_1}\dots (x - \alpha_k)^{d_k}$, is the factorization of $f(x)$ in the algebraic closure of $\F_q$,  
    where $\alpha_i \neq \alpha_j$  (for $i \neq j$), then gcd$(d, d_1, d_2, \dots, d_k) = 1$.
\end{enumerate}
\end{lem}

\begin{thm}[Weil's bound \cite{Weil}] \label{Weil}
Let $q$ be a prime power and $\chi \in \widehat{\F_q^*}$, with order $|\chi| = d > 1$. Let $\psi$ be an additive character of $\F_q$. Let $m, n \geq 1$ and let $f, g \in \F_q[x]$ with the following properties:
\begin{enumerate}
    \item The polynomial $f$ has $m$ distinct roots in the algebraic closure of $\F_q$.
    \item The polynomial $g$ is of degree $n$. 
    \item The polynomials $y^d - f(x)$ and $z^q - z - g(x)$ are absolutely irreducible.
\end{enumerate}
Then
$$
\left| \sum_{x\in \F_q}\chi(f(x))\psi(g(x)) \right| \leq (m+n - 1)\sqrt{q}.
$$
\end{thm}

\begin{cor}\label{MainTool}
Let $q$ be an odd prime power, $c \in \F_q$ be a constant, $\chi \in \widehat{\F_q^*}$ be a nontrivial character, and $\delta_{xy}$ be the Kronecker delta function on $\F_q$. Then
$$
\left| \sum_{t \in \F_q} \chi(t)\eta(t^2 - 1)\psi(ct) \right| \leq (3 - \delta_{0c})\sqrt{q}.
$$
\end{cor}

\begin{proof}
Let $\chi \in \widehat{\F_q^*}$ with $|\chi| = d > 1$. Let $\chi_1$ be a primitive character such that $\chi(t) = \chi_1(t^{(q-1)/d})$. If $s = \text{gcd}((q-1)/d, (q-1)/2)$, then by Lemma \ref{Schlem} and the Weil's bound (Theorem \ref{Weil})
\begin{align*}
&\left| \sum_{t \in \F_q} \chi(t)\eta(t^2-1)\psi(ct) \right| = \left| \sum_{t \in \F_q}\chi_1^s\left(t^{(q-1)/ds}(t^2 - 1)^{(q-1)/2s}\right)\psi(ct) \right| \leq (3 - \delta_{0c})\sqrt{q}. \qedhere
\end{align*}

\end{proof}

\begin{cor}\label{MainTool2}
Let $q$ be an odd prime power, $c \in \F_q^*$ be a constant, and $\chi \in \widehat{\F_q^*}$ be a nontrivial character. Then
$$
\left| \sum_{\substack{t\in \mathbb{F}_q\\ t\neq -1}}\chi(t)\eta(t^2 - 1)\zeta^{-tr(c\frac{(t - 1)}{(t + 1)})} \right| \leq 3\sqrt{q}.
$$
\end{cor}

\begin{proof}
Note that, for $t \neq -1$, we have that
$
\eta(t^2 - 1) = \eta\left(\frac{t-1}{t+1}\right).
$
Furthermore
$$
s = \frac{t-1}{t+1} \iff t = -\frac{s+1}{s-1}.
$$
Thus, by the same reasoning as in Corollary \ref{MainTool}, Weil's bound gives
$$
\left|\sum_{\substack{t\in \mathbb{F}_q\\ t\neq -1}}\chi(t)\eta(t^2 - 1)\zeta^{-tr(c\frac{(t - 1)}{(t + 1)})} \right|= \left| \sum_{s \in \F_q}\chi(-(s-1)^{(q-2)}(s+1))\eta(s)\zeta^{-tr(cs)}\right| \leq 3\sqrt{q}.
$$
\end{proof}

\section{Irreducible representations of the group \boldmath{$G$}}\label{Section_irr_rep}
For background on representation theory, and its applications to compute eigenvalues of Cayley graphs we refer the reader to \cite[Ch.\ 3]{DSV}, and \cite[Part 3]{KS}. Let $G = (\F_q^5, \cdot)$, where $\cdot$ is as defined in (\ref{GroupOp}). Denote $X= (x_1,x_2,x_3,x_4,x_5) \in G$ and $Y= (y_1,y_2,y_3,y_4,y_5) \in G$. For each $\alpha, \beta, \gamma \in \mathbb{F}_q$, the linear characters of $G$ are given by
\begin{align}\label{Linear_char}
\chi_{\alpha,\beta,\gamma}(X) := \zeta^{tr(\alpha x_1 + \beta x_2 + \gamma x_3)}.
\end{align}
For each $\alpha, \beta, \gamma \in \mathbb{F}_q$ with $\alpha \neq 0$, we define the dimension $q$ representation $M_{\alpha,\beta,\gamma}: G \to GL_q(\mathbb{C})$ by
\begin{align*}
M_{\alpha,\beta,\gamma}(X) := [\zeta^{tr\left(\left(x_2+\frac{\beta}{\alpha}x_3\right)j +  \alpha x_4 + \beta x_5 + \gamma x_3\right)} \delta_{2x_1\alpha + j, k}]_{j, k \in \F_q}.
\end{align*}
 For each $\tau, \mu \in \mathbb{F}_q$ with $\tau \neq 0$, we define the dimension $q$ representation $N_{\tau,\mu} : G \to GL_q(\mathbb{C})$ by
\begin{align*}
N_{\tau,\mu}(X) := [\zeta^{tr(  x_3 j +\tau x_5 + \mu x_2 )} \delta_{2x_1\tau+j,k}]_{j, k \in \F_q}.
\end{align*}

It is easy to check that
 $M_{\alpha,\beta,\gamma}(X Y) = M_{\alpha,\beta,\gamma}(X)M_{\alpha,\beta,\gamma}(Y)$ and 
  $N_{\tau,\mu}(X Y) = N_{\tau,\mu}(X)N_{\tau,\mu}(Y)$,
i.e., they are group homomorphisms, and so are group representations of $G$. The associated characters are found to be
\begin{align*}
\psi_{\alpha,\beta,\gamma}(X) &:= \text{Tr }(M_{\alpha,\beta,\gamma}(X))
= \begin{cases}
q \zeta^{tr(\alpha x_4 + \beta x_5 + \gamma x_3)}, &\text{if } x_1 = 0,\text{and } x_2 = -\frac{\beta}{\alpha}x_3\\
0, &\text{otherwise};
\end{cases}\\
\phi_{\tau,\mu}(X) &:= \text{Tr }(N_{\tau,\mu}(X))
= \begin{cases}
q \zeta^{tr(\tau x_5 + \mu x_2)}, &\text{if } x_1 = 0,\text{and } x_3 = 0\\
0, &\text{otherwise}.
\end{cases}
\end{align*} 
These $q^3-q$ characters of $G$ are irreducible and inequivalent since
\begin{align*}
[\psi_{\alpha,\beta,\gamma},\psi_{\alpha',\beta',\gamma'}]_G &= \frac{1}{q^5} \sum_{x_1,x_2,x_3,x_4,x_5\in \mathbb{F}_q} \psi_{\alpha,\beta,\gamma}(X)\overline{\psi_{\alpha',\beta',\gamma'}(X)}\\
&= \frac{1}{q^3} \sum_{x_3 \in \mathbb{F}_q} \zeta^{tr((\gamma-\gamma')x_3)}\sum_{x_4\in \mathbb{F}_q} \zeta^{tr((\alpha -\alpha')x_4)} \sum_{x_5\in \mathbb{F}_q} \zeta^{tr((\beta-\beta')x_5)}\\
&= \begin{cases}
1, &\text{if } (\alpha,\beta,\gamma) = (\alpha',\beta',\gamma')\\
0, &\text{otherwise};
\end{cases}\\
\end{align*}
\begin{align*}
[\phi_{\tau, \mu},\phi_{\tau',\mu'}]_G &= \frac{1}{q^5} \sum_{x_1,x_2,x_3,x_4,x_5\in \mathbb{F}_q} \phi_{\tau, \mu}(X)\overline{\phi_{\tau',\mu'}(X)}\\
&= \frac{1}{q^3} \sum_{x_4\in \mathbb{F}_q}\sum_{x_5\in \mathbb{F}_q} \zeta^{tr((\tau-\tau')x_5)}\sum_{x_2\in \mathbb{F}_q}\zeta^{tr((\mu-\mu')x_2)}\\
&= \begin{cases}
1, &\text{if } (\tau,\mu) = (\tau',\mu')\\
0, &\text{otherwise};
\end{cases}
\end{align*}
\begin{align*}
[\psi_{\alpha,\beta,\gamma},\phi_{\tau,\mu}]_G &= \frac{1}{q^5}\sum_{x_1,x_2,x_3,x_4,x_5\in \mathbb{F}_q} \psi_{\alpha,\beta,\gamma}(X)\overline{\phi_{\tau,\mu}(X)}\\
&= \frac{1}{q^3} \sum_{x_5\in \mathbb{F}_q}\zeta^{tr((\beta-\tau)x_5)}\sum_{x_4\in \mathbb{F}_q}\zeta^{tr(\alpha x_4)} = 0.
\end{align*}

Furthermore, these are all the irreducible representations since we have that the sum of the squares of the dimensions of all irreducible representations should equal the order of the group, $(q^3 - q)q^2 + q^3 = q^5$.

Since the point graph $\Gamma^{(2)}(q)$ of $D(5,q)$ is $\text{Cay}(G,S)$, then the adjacency matrix $A$ of $\Gamma^{(2)}(q)$ has characteristic polynomial
\begin{align*}
\Phi(x) &=  \prod_{\alpha,\beta,\gamma\in \mathbb{F}_q}(x-\chi_{\alpha,\beta,\gamma}(S)) \prod_{\substack{\alpha,\beta,\gamma \in \mathbb{F}_q\\ \alpha\neq 0}} \text{det}[xI_q-M_{\alpha,\beta,\gamma}(S)]^q \prod_{\substack{\tau,\mu \in \mathbb{F}_q\\ \tau \neq 0}} \text{det}[xI_q - N_{\tau,\mu}(S)]^q
\end{align*}
where 
$$
\chi_{\alpha, \beta, \gamma}(S) := \sum_{g \in S} \chi_{\alpha, \beta, \gamma}(g), \hs M_{\alpha, \beta, \gamma}(S) := \sum_{g \in S}M_{\alpha, \beta, \gamma}(g), \hs N_{\tau, \mu}(S) := \sum_{g \in S}N_{\tau, \mu}(g).
$$
The matrices $M_{\alpha, \beta, \gamma}(S)$ and $N_{\tau, \mu}(S)$ will always be assumed to be indexed by the elements of $\F_q$. 
Note that the contribution of the eigenvalues from the linear characters are the same as that of $D(4, q)$, see \cite[Section 6]{MSW}. So we are left with determining the eigenvalues coming from the $M_{\alpha,\beta, \gamma}(S)$ (Section \ref{MMatrix}) and $N_{\tau, \mu}(S)$ (Section \ref{NMatrix}). 

\section{Spectrum of matrices \boldmath{$M_{\alpha, \beta, \gamma}(S)$}}\label{MMatrix}

In this section, we determine the exact spectrum of each matrix $M_{\alpha, \beta, \gamma}(S)$. Furthermore, by application of Corollaries \ref{MainTool} and \ref{MainTool2}, we will demonstrate that each eigenvalue has modulus at most $3q$. The insight we can share ahead of the proofs is that each matrix $M_{\alpha, \beta, \gamma}(S)$ is similar to a matrix that has $q-2$ eigenvectors of the form $v_\chi$, where $(v_\chi)_j = (\chi(j))$ for each non-trivial character $\chi \in \widehat{\F_q^*}$. The remaining two eigenvectors have the form $(z, 1, 1, \dots, 1)$ where $z$, and respectively the corresponding eigenvalue $\lambda$, satisfy a particular quadratic polynomial.

We begin with evaluating $M_{\alpha, \beta, \gamma}(S)$.
\begin{align*}
M := M_{\alpha,\beta,\gamma}(S) &= \sum_{g\in S}M_{\alpha,\beta,\gamma}(g)\\
&= \sum_{\substack{a,x\in \mathbb{F}_q\\ x\neq 0}} M_{\alpha,\beta,\gamma}((x,xa,xa^2,x^2a,x^2a^2))\\
&=\sum_{\substack{x\in \mathbb{F}_q\\ x\neq 0}}\sum_{a\in \mathbb{F}_q} [\zeta^{tr\left((xa+\frac{\beta}{\alpha}xa^2)j+\alpha x^2a +\beta x^2 a^2 +\gamma xa^2\right)}\delta_{2x\alpha+j,k}]_{j, k \in \F_q}.
\end{align*}
 So in particular, the diagonal entries are zero and we have that for each $j, x \in \F_q$ with $x \neq 0$,
\begin{align}\label{Sum1}
M_{j, 2\alpha x + j} &= \sum_{a\in \mathbb{F}_q} \zeta^{tr\left((xa+\frac{\beta}{\alpha}xa^2)j+\alpha x^2a +\beta x^2 a^2 +\gamma xa^2\right)} \nonumber \\
&= \sum_{a\in \mathbb{F}_q} \zeta^{tr\left(a(xj + \alpha x^2) + a^2(\frac{\beta}{\alpha}x j + \beta x^2 + \gamma x) \right)} \\
&= \sum_{a\in \mathbb{F}_q} \zeta^{tr\left(a f(x, j) + a^2 g(x, j) \right),} \nonumber
\end{align}
where $f(x, j) = xj + \alpha x^2$ and $g(x, j) = \frac{\beta}{\alpha}xj + \beta x^2 + \gamma x = \frac{\beta}{\alpha}f(x) + \gamma x$.


Now if $f(x, j) = g(x, j) = 0$, then (\ref{Sum1}) evaluates to $q$. If $g(x, j) = 0$ and $f(x, j) \neq 0$, then (\ref{Sum1}) evaluates to 0. If $g(x, j) \neq 0$, then we can rearrange the exponents and evaluate the sum. Note
\begin{align*}
af(x, j) + a^2g(x, j) = g(x,j)\left(a + \frac{f(x, j)}{2g(x, j)} \right)^2 - \frac{f(x, j)^2}{4g(x, j)}.
\end{align*}
Therefore we have 
\begin{align*}
    M_{j, 2\alpha x + j} &= \sum_{a\in \mathbb{F}_q} \zeta^{tr\left(a f + a^2 g \right)} 
    = \sum_{a\in \mathbb{F}_q} \zeta^{tr\left(g\left(a + \frac{f}{2g} \right)^2 - \frac{f^2}{4g} \right)} 
    = {\zeta}^{tr(-\frac{f^2}{4g})}\sum_{a\in \mathbb{F}_q}\zeta^{tr(ga^2)} = \zeta^{-tr(\frac{f^2}{4g})}\eta(g)S_{y^2}.
\end{align*}
Letting $k = 2\alpha x + j$, and let us perform a change of variables to obtain $M_{j, k}$ in terms of just $j$ and $k$.
We obtain
$$
F_{\alpha}(j, k) :=f\left( \frac{k - j}{2\alpha}, j\right) = \frac{1}{4\alpha}(k^2 - j^2),
$$
and
$$
G_{\alpha, \beta, \gamma}(j, k) := g\left(\frac{k - j}{2\alpha}, j\right) = \frac{\beta}{4\alpha^2}\left( k^2 - j^2 +\frac{2\gamma\alpha (k - j)}{\beta} \right).
$$
The above discussion allows us to define the entries of $M$ explicitly. We summarize this as a lemma.
\begin{lem}\label{MM}
Let $M = M_{\alpha, \beta, \gamma}(S)$, $F = F_{\alpha}(j, k)$, and $G = G_{\alpha, \beta, \gamma}(j, k)$. Then
 \[
 M_{j, k} = 
 \left\{
 \begin{array}{ll}
\eta(G){\zeta}^{-tr(\frac{F^2}{4G})}S_{y^2}, & \mbox{\text{if } $G \neq 0$}, \\
 0, & \mbox{\text{if } $G = 0$ and $F \neq 0$ } \\
 0, & \mbox{\text{if } $j = k$} \\
 q, & \mbox{\text{if otherwise}}.
 \end{array}
 \right.
 \]
\end{lem}

\begin{remark}
When $\beta = 0$, we may instead directly substitute $x = \frac{k-j}{2\alpha}$ in (\ref{Sum1}) to see that for each pair $\alpha, \gamma \in \F_q$ with $\alpha \neq 0$, $M_{-\frac{1}{4\alpha},0,-\frac{\gamma}{2\alpha}}(S)$ is the same matrix as that of in \cite[Section 6]{MSW}, whose spectrum was determined and was shown to be upper bounded by $3q$. Thus, for $M_{\alpha,\beta,\gamma} (S)$ we assume that $\beta \neq 0$.
\end{remark}

Having now an explicit description of the entries in $M$, we are left with analyzing the exponent of $\zeta$ in $M_{j, k}$ to better understand the matrices. We have
$$
-\frac{F(j,  k)^2}{4G(j, k)} = -\frac{\left( \frac{1}{4\alpha}(k^2 - j^2)\right)^2}{4\left(\frac{\beta}{4\alpha^2}\left( k^2 - j^2 +\frac{2\gamma\alpha (k-j)}{\beta} \right)  \right)}= - \frac{1}{16\beta}\left( \frac{(k^2 - j^2)^2}{k^2 - j^2 + 2\gamma\alpha (k - j)/\beta} \right).
$$

The following lemma allows us to reduce the problem of finding the spectrum of $M_{\alpha, \beta, \gamma}(S)$ to only concerning ourselves with the spectrum of $M_{1, \beta, \gamma}$.

\begin{lem}
Let $\alpha, \beta \in \F_q^*$, $ \gamma \in \F_q$.Then
$$
M_{\alpha, \beta, \gamma} = M_{1, \beta, \gamma\alpha}.
$$
\end{lem}

\begin{proof}
Let $M = M_{\alpha, \beta, \gamma}(S)$ and $M' = M_{1, \beta, \alpha\gamma}(S)$.
\begin{enumerate}
    \item If $M_{j, k} = 0$, then either $j = k$, or $G_{\alpha, \beta, \gamma}(j, k) = 0$ and $F_{\alpha}(j, k) \neq 0$.
    If $j = k$, then $M'_{j, j} = 0 = M_{j, j}$.
    If $G_{\alpha, \beta, \gamma}(j, k) = 0$ and $F_{\alpha}(j, k) \neq 0$ then note, 
    $$
    G_{\alpha, \beta, \gamma}(j ,k) = 0 \iff (k^2 - j^2 + 2\gamma\alpha (k - j)/\beta) =0 \iff G_{1, \beta, \alpha\gamma}(j ,k) = 0,
    $$
    and 
    $$
    F_\alpha(j, k) \neq 0 \iff k^2 - j^2 \neq 0 \iff F_1(j, k) \neq 0
    $$
    and therefore $M'_{j, k} = 0$.
    \item If $M_{j, k} = q$, then $F_{\alpha}(j, k) = G_{\alpha, \beta, \gamma}(j, k) =  0$ and $j \neq k$. This can only occur if $\gamma = 0$. But then $-F_{\alpha}(j, k)^2/4G_{\alpha, \beta, 0}(j, k)  = -(k^2 - j^2)/16\beta$ which is independent of $\alpha$ and so $M_{j, k} = M'_{j, k}$.
    \item If $M_{j, k} \neq 0, q$, then
    $$
    -\frac{F_{\alpha}(j,  k)^2}{4G_{\alpha, \beta, \gamma}(j, k)} = - \frac{1}{16\beta}\left( \frac{(k^2 - j^2)^2}{k^2 - j^2 + 2\gamma\alpha (k - j)/\beta} \right) = -\frac{F_{1}(j,  k)^2}{4G_{1, \beta, \alpha\gamma}(j, k)}.
    $$
    Furthermore, we note that $\alpha^2G_{\alpha, \beta, \gamma} = G_{1, \beta, \alpha\gamma}$ and so 
    $
    \eta(G_{\alpha, \beta, \gamma})   = \eta(G_{1, \beta, \alpha, \gamma})
    $
    So $M_{j, k} = M'_{j, k}$ in this final case too.
\end{enumerate}
Thus, $M_{\alpha, \beta, \gamma}(S) = M_{1, \beta, \alpha\gamma}(S)$.
\end{proof}

\subsection{Spectrum of \boldmath{$M_{1, \beta, 0}(S)$}}\label{Subsection_5.1}

We will now carefully detail how to determine the eigenvalues of $M_{1, \beta, 0}(S)$. The case of $M_{1, \beta, \gamma}(S)$ with $\gamma \neq 0$ is very similar, as are the cases for $N_{\tau, \mu}(S)$ which are done in the following section. 

Consider now the matrix $M = M_{1, \beta, 0}(S)$ where $\beta \in \F_q^*$. Recall that 
$$
G_{1, \beta, 0}(j, k) = \frac{\beta}{4}(k^2 - j^2)
$$
and so $\eta(G_{1, \beta, 0}(j, k)) = 1$ if and only if $\eta(\beta(k^2 - j^2)) =1$. Furthermore, observe that the argument of the trace function in the exponent of any $\zeta$ term in $M$ is 
$$
 -\frac{F_{1}(j, k)^2}{4G_{1, \beta, 0}(j, k)}= - \frac{1}{16\beta}\left( \frac{(k^2 - j^2)^2}{k^2 - j^2} \right) = \frac{1}{16\beta}(j^2 - k^2).
$$
We may decompose $M$, as given in Lemma \ref{MM}, into the product $S_{y^2}(D^*U_{\beta}D)$ where $D$ is a diagonal matrix with 
$$
D_{j, j}= \zeta^{tr(\frac{j^2}{16\beta})}
$$
and $U_{\beta}$ is given as follows:
    \[
 (U_{\beta})_{j, k} = 
 \left\{
 \begin{array}{ll}
 \eta(-1)S_{y^2} & \mbox{\text{if }$j = -k$, and $j \neq 0$}. \\
\eta(\beta(k^2 - j^2)) & \mbox{\text{otherwise}}. 
 \end{array}
 \right.
 \]
In particular, we bring attention to the fact that $\eta(-1)S_{y^2} = q/S_y^2$.

\begin{lem}\label{Lem54}
Let $q = p^e$ be an odd prime power. Let $\beta \in \F_q^*$ and $\chi \in \widehat{\F_q^*}$, where $|\chi| = r > 1$. Then the vector $v_\chi$ whose coordinates are indexed by $j \in \F_q$, given by
$
(v_\chi)_j = (\chi(j))
$
is an eigenvector of $U_{\beta}$ with corresponding eigenvalue
$$
  \lambda_\chi =  \eta(\beta)\left( \sum_{t \in \F_q} \chi(t)\eta(t^2- 1)\right) +(-1)^{(q-1)/r} \eta(-1)S_{y^2}.
$$
Furthermore, $|\lambda_\chi| \leq 3\sqrt{q}$.
\end{lem}
\begin{proof}
Let $U = U_{\beta}$ and let us compute $(Uv_\chi)_j$. When $j = 0$, it is clear that $(Uv_\chi)_0 = 0$ since we obtain just the sum of all $(q-1)$st roots of unity. When $j \neq 0$, we have 
    \begin{align*}
    (Uv_\chi)_j &= \left(  \sum_{k \in \F_q^*} \chi(k)\eta(\beta(k^2 - j^2))\right) + \chi(-j)\eta(-1)S_{y^2}\\ 
    &= \chi(j)\left(  \eta(\beta)\left(     \sum_{t \in \F_q^*} \chi(t)\eta(t^2 - 1)\right) + \chi(-1)\eta(-1)S_{y^2}\right).
    \end{align*}
    We note that $\chi(-1) = (-1)^{(q-1)/r}$ where $r$ is the order of $\chi$. Therefore $v_\chi$ has corresponding eigenvalue 
    $$
    \lambda_\chi =  \eta(\beta) \left( \sum_{t\in \F_q} \chi(t)\eta(t^2 - 1)\right) + (-1)^{(q-1)/r} \eta(-1)S_{y^2}.
    $$
Applying Corollary \ref{MainTool}, we immediately obtain $|\lambda_\chi| \leq 3\sqrt{q}$.
\end{proof}
The above theorem gives us $q - 2$ linearly independent eigenvectors and their corresponding eigenvalues. So we are now left with determining the remaining two eigenvalues of $U_{\beta}$.

\begin{lem}\label{Lem55}
Let $q = p^e$ be an odd prime power, $\beta \in \F_q^*$. $U_{\beta}$ has 2  eigenvectors of the form 
$
(z, 1, 1, \dots, 1).
$
Furthermore the corresponding eigenvalues satisfy:
\begin{enumerate}[(i)]
    \item
    $
\lambda^2 - (S_{y^2} -2 \eta(\beta))\lambda  - (q - 1) = 0
$
if $q \equiv 1$(mod 4).
\item 
    $
\lambda^2 + S_{y^2}\lambda  + (q - 1) = 0
$
if $q \equiv 3$(mod 4).
\end{enumerate}
Furthermore, each eigenvalue satisfies $|\lambda| < 3\sqrt{q}$.
\end{lem}

\begin{proof}
Let $v = (z, 1, 1, \dots, 1)$, where $z$ is a variable. We will show that $v$ is an eigenvector by solving for $z$ and the corresponding eigenvalue $\lambda$ simultaneously. The product of first row of $U_{\beta}$ (when $j = 0$) with $v$ gives $z \lambda = \eta(\beta)(q-1)$. For a fixed $j \in \F_q^*$, the product of the row of $U_\beta$ indexed by $j$, with $v$ gives
\begin{align} \label{Eq1}
\lambda &= \eta(-\beta)z + \sum_{k \in \F_q^*} \eta(\beta(k^2 - j^2)) + \eta(-1)S_{y^2} = \eta(-\beta)z + \eta(\beta)\left(\sum_{t \in \F_q^*}\eta(t^2 - 1)\right) + \eta(-1)S_{y^2}.
\end{align}
Lemma \ref{LemSq} computes the sum in (\ref{Eq1}). Thus
\begin{enumerate}[(i)]
    \item If $q \equiv 1$(mod 4), then 
    $
    \lambda = \eta(-\beta)z -2 \eta(\beta) + S_{y^2}.
    $
    \item If $q \equiv 3$(mod 4), then 
    $
    \lambda = \eta(-\beta)z - S_{y^2}.
    $
\end{enumerate}
Recalling that $\lambda z = \eta(\beta)(q- 1)$ allows us to solve for both $\lambda$ and $z$. Using the quadratic formula, one easily obtains that each $\lambda$ satisfies $|\lambda| < 3\sqrt{q}$.
\end{proof}

%

\subsection{Spectrum of \boldmath{$M_{1,\beta,\gamma}(S)$}}\label{Subsection_5.2}
We now turn to determining the spectrum of $M_{1, \beta, \gamma}(S)$ when $\gamma \neq 0$. The following exposition is very similar to that of the case when $\gamma = 0$ and so we will do our best not to repeat ourselves. Let $M = M_{1, \beta, \gamma}(S)$ with $\gamma \neq 0$. Recall that
$$
G_{1, \beta, \gamma}(j, k) = \frac{\beta}{4}\left( k^2 - j^2 +\frac{2\gamma(k - j)}{\beta} \right)  = \frac{\beta}{4}\left(\left( k + \frac{\gamma}{\beta} \right)^2  - \left(j + \frac{\gamma}{\beta}  \right)^2 \right),
$$
and 
$$
F_{1}(j, k) = \frac{1}{4}(k^2 - j^2).
$$
It is clear that $(M)_{j, k}$ is similar to the matrix $(M')_{j, k} = (M)_{j - \gamma/\beta, k - \gamma/\beta}$. Working now with $M'$ instead, we define
$$
G' = G'_{1, \beta, \gamma}(j, k) = G_{1, \beta, \gamma}(j - \gamma/\beta, k - \gamma/\beta), \hs \text{and} \hs F' = F'_{1} = F_1(j - \gamma/\beta, k - \gamma/\beta).
$$
Note that when $\gamma, \beta \neq 0$, then $G' = F' = 0$ only occurs when $j = k$. Thus, by Lemma \ref{MM} we have

 \[
 M'_{j, k} =  M_{j - \gamma/\beta, k - \gamma/\beta} =
 \left\{
 \begin{array}{ll}
  0 & \mbox{if $j^2 = k^2$}, \\
\eta(G'){\zeta}^{-tr(\frac{F'^2}{4G'})}S_{y^2} & \mbox{otherwise}.
 \end{array}
 \right\} = \eta(\beta(k^2 - j^2)){\zeta}^{-tr(\frac{F'^2}{4G'})}S_{y^2}.
 \]
The argument of the trace function in the exponent of the non-zero terms is of the form
\begin{align*}
-\frac{F'^2}{4G'} &= 
- \frac{1}{16\beta}\left( \frac{(k^2 - j^2 - 2\frac{\gamma}{\beta} (k - j))^2}{k^2 - j^2} \right) \\ &= -\frac{1}{16\beta} \left( \frac{(k^2 - j^2)^2}{k^2 - j^2} -4\frac{\gamma}{\beta}\frac{(k^2 - j^2)(k - j)}{k^2 - j^2} + 4\frac{\gamma^2}{\beta^2} \frac{(k-j)^2}{k^2 - j^2}\right) \\
&=-\frac{1}{16\beta}(k^2 - j^2) +\frac{\gamma}{4\beta^2}(k-j) - \frac{\gamma^2}{4\beta^3}\frac{k-j}{k+j}.
\end{align*}
Finally, let $D$ be the diagonal matrix with 
$$
D_{j, j} = \zeta^{-tr(j^2/16\beta - \gamma j/4\beta^2)}.
$$
That way we have that $M' = S_{y^2}(D^*W_{\beta, \gamma} D)$, where 
\begin{align}\label{Eq7}
(W_{\beta, \gamma})_{j, k} =  \left\{
 \begin{array}{ll}
  0, & \mbox{\text{if }$j^2 = k^2$}, \\
\eta(\beta(k^2 - j^2))\zeta^{-tr(\frac{\gamma^2(k - j)}{4\beta^3(k + j)})}, & \mbox{\text{otherwise}}. 
\end{array}\right.
\end{align}
Observe that $W_{\beta, \gamma}$ also have the same $q -2$ eigenvectors, $v_\chi$, as in the case when $\gamma = 0$.

\begin{lem}\label{Lem5}
Let $q = p^e$ be an odd prime power. Let $\beta \in \F_q^*$ and $\chi \in \widehat{\F_q^*}$, where $|\chi| = r > 1$. Then the vector $v_\chi$ whose coordinates are indexed by $j \in \F_q$, given by
$
(v_\chi)_j = (\chi(j))
$
is an eigenvector of $W_{\beta, \gamma}$ with corresponding eigenvalue
\begin{align*}\label{Eq4}
  \lambda_\chi =  \eta(\beta)\left( \sum_{\substack{t\in \mathbb{F}_q\\ t\neq -1}} \chi(t)\eta(t^2- 1)\zeta^{-tr(\frac{\gamma^2(t - 1)}{4\beta^3(t + 1)})}\right).
\end{align*}
Furthermore, $|\lambda_\chi| \leq 3\sqrt{q}.$
\end{lem}

\begin{proof}
Let $W = W_{\beta, \gamma}$. Computing $Wv_\chi$, we obtain immediately that when $j = 0$, $(Wv_\chi)_j = 0$. For a fixed $j \in \F_q^*$, we compute
\begin{align*}
&(Wv_\chi)_j = \sum_{\substack{k\in \mathbb{F}_q\\ k\neq -j}}\chi(k)\eta(\beta(k^2 - j^2))\zeta^{-tr(\frac{\gamma^2(k - j)}{4\beta^3(k + j)})} = \chi(j)\left(\eta(\beta)\sum_{\substack{t\in \mathbb{F}_q\\ t\neq -1}}\chi(t)\eta(t^2 - 1)\zeta^{-tr(\frac{\gamma^2(t - 1)}{4\beta^3(t + 1)})}\right). 
\end{align*}
Applying Corollary \ref{MainTool2}, we immediately obtain the desired bound.
\end{proof}
\begin{lem}\label{Lem6}
Let $q = p^e$ be an odd prime power, $\beta \in \F_q^*$. $W_{\beta}$ has 2  eigenvectors of the form 
$
(z, 1, 1, \dots, 1).
$
Furthermore the corresponding eigenvalues satisfy:
\begin{enumerate}[(i)]
    \item
    $
\lambda^2 - (S_{y^2} -2\cos(\frac{2\pi}{p}tr(\frac{\gamma^2}{4\beta^3})))\lambda  - (q - 1) = 0
$
if $q \equiv 1$(mod 4).
\item 
    $
\lambda^2 - (S_{y^2} - 2\cos(\frac{2\pi}{p}tr(\frac{\gamma^2}{4\beta^3})))\lambda  + (q - 1) = 0
$
if $q \equiv 3$(mod 4).
\end{enumerate}
Furthermore, each eigenvalue satisfies $|\lambda| < 3\sqrt{q}$.
\end{lem}
\begin{proof}
Let $W = W_{\beta, \gamma}$ and $v = (z, 1, 1, \dots, 1)$ be an eigenvector of $W$. We will solve for the corresponding eigenvalue $\lambda$ and $z$ simultaneously.  The product of first row of $W_{\beta}$ (when $j = 0$) with $v$ gives  $\eta(\beta)\zeta^{-tr(\gamma^2/4\beta^3)}(q-1) = z \lambda$. Denote $c = \eta(-\beta)\zeta^{tr(\gamma^2/4\beta^3)}$.  For a fixed $j \in \F_q^*$, the product of the row of $W$ indexed by $j$, with $v$ gives
$$
(W v)_j =cz + \sum_{\substack{k \in \F_q \\ k \neq 0, -j}} \eta(\beta(k^2 - j^2))\zeta^{-tr(\frac{\gamma^2(k-j)}{4\beta^3(k+j)})} = cz + \eta(\beta)\left(\sum_{\substack{t \in \F_q \\ t \neq 0, -1}}\eta(t^2 - 1)\zeta^{-tr(\frac{\gamma^2(t-1)}{4\beta^3(t+1)})}\right).
$$
We note that for $t \neq -1$, $\eta(t^2 - 1) = \eta(\frac{t-1}{t+1})$ and so by performing the change of variable $s = \frac{t-1}{t+1}$, we have that
\begin{equation}\label{eqnum1}
(W v)_j =  cz + \eta(\beta)\sum_{\substack{s\in \mathbb{F}_q\\ s \neq -1, 1}}\eta(s)\zeta^{-tr(\frac{\gamma^2}{4\beta^3}s)} = cz + S_{y^2} - (\zeta^{tr(\frac{\gamma^2}{4\beta^3})} + \zeta^{-tr(\frac{\gamma^2}{4\beta^3})})  =  \lambda.
\end{equation}
Recalling that $\lambda = \eta(\beta)\zeta^{-tr(\gamma^2/4\beta^3)}(q-1)/z$ allows us to solve for $\lambda$. Furthermore, the quadratic formula gives that the eigenvalues are bounded  by $3\sqrt{q}$.
\end{proof}
Thus we have shown that for each matrix $M_{\alpha, \beta, \gamma}(S)$, the corresponding matrix $U_\beta$ or $W_{\beta, \gamma}$ we constructed has all eigenvalues bounded by $3\sqrt{q}$. Recall that either $M_{\alpha, \beta, \gamma}(S) = S_{y^2}(D^*U_\beta D)$ or $M_{\alpha, \beta, \gamma}(S) = S_{y^2}(D^*W_{\beta, \gamma} D)$, meaning if  $\lambda$ is an eigenvalue of $U_\beta$ (or $W_{\beta, \gamma}$ respectively), then $S_{y^2}\lambda$ is an eigenvalue of $M_{\alpha, \beta, \gamma}(S)$. So we conclude that if $\lambda$ is an eigenvalue of $M_{\alpha, \beta, \gamma}(S)$, then $|\lambda| \leq 3q$.
\section{Spectrum of matrices \boldmath{$N_{\tau, \mu}(S)$}}\label{NMatrix}

As we will show, the matrices $N_{\tau, \mu}(S)$ are very much akin to $M_{\alpha, \beta, \gamma}(S)$. With that in mind, we will avoid duplicating computations. In this section, our goal is to determine the spectrum of matrices $N_{\tau, \mu}(S)$ and demonstrate that each eigenvalue of $N_{\tau, \mu}$ satisfies $|\lambda| \leq 3q$. Recall that 
$$
N_{\tau, \mu}(S) = \sum_{a\in \mathbb{F}_q}\sum_{\substack{x\in \mathbb{F}_q\\ x\neq 0}} [\zeta^{[tr(xa^2 j+\tau x^2a^2 +\mu xa)]}\delta_{2x\tau +j, k}]_{j, k\in \F_q}.
$$
We can follow the same steps as for the  computation for $M_{\alpha, \beta, \gamma}(S)$. Let $f(x, j) = \mu x$ and $g(x, j) = xj + \tau x^2$.
Then we may set 
$$
F_{\tau, \mu}(j, k) = f\left(\frac{k - j}{2\tau}, j\right) = \frac{\mu(k - j)}{2\tau}, \hs 
G_\tau(j, k) = g\left(\frac{k  - j}{2\tau}, j\right) = \frac{k^2 - j^2}{4\tau}.
$$

\begin{lem}
Let $N = N_{\tau, \mu}(S)$, $F = F_{\tau, \mu}(j, k)$, and $G = G_{\tau}(j, k)$. Then
 \[
 N_{j, k} = 
 \left\{
 \begin{array}{ll}
\eta(G){\zeta}^{-tr(\frac{F^2}{4G})}S_{y^2}, & \mbox{\text{if } $G \neq 0$}, \\
 0, & \mbox{\text{if } $G = 0$ and $F \neq 0$ } \\
 0, & \mbox{\text{if } $j = k$} \\
 q, & \mbox{\text{if otherwise}}.
 \end{array}
 \right.
 \]
\end{lem}

If $\mu = 0$, then $F_{\tau, \mu}(j, k) = 0$. Then we have that 
\begin{align*}
\frac{1}{S_{y^2}}N_{j, k} = \left\{ 
\begin{array}{ll}
    \eta(-1)S_{y^2} & \mbox{if $j = -k$ and $j \neq 0$}  \\
    \eta(\tau(k^2 - j^2) & \mbox{otherwise}
\end{array}  \right.
\end{align*}

We note that $\frac{1}{S_{y^2}}N_{\tau, 0}$ is precisely the matrix $U_\tau$ in section 5, whose eigenvalues were determined. So we immediately obtain the following.

\begin{lem}
Let $q = p^e$ be an odd prime power, $\tau \in \F_q^*$. Then the eigenvalues of $\frac{1}{S_{y^2}}N_{\tau, 0}(S)$ are:
\begin{enumerate}
    \item For each non-trivial $\chi \in \widehat{\F_q^*}$,   $$  \lambda_\chi =  \eta(\tau)\left( \sum_{t \in \F_q} \chi(t)\eta(t^2- 1)\right) +(-1)^{(q-1)/r} \eta(-1)S_{y^2}.
    $$
    \item  $\lambda_1, \lambda_2$, where $\lambda_1, \lambda_2$ are the roots of 
    \begin{enumerate}[(i)]
    \item
    $
\lambda^2 - (S_{y^2} -2 \eta(\tau))\lambda  - (q - 1) = 0
$
if $q \equiv 1$(mod 4).
\item 
    $
\lambda^2 + S_{y^2}\lambda  + (q - 1) = 0
$
if $q \equiv 3$(mod 4).
\end{enumerate}
\end{enumerate}
Furthermore, each eigenvalue $\lambda$ satisfies $|\lambda| < 3\sqrt{q}$.
\end{lem}
 
 When $\mu \neq 0$, then 
 \begin{align}\label{Eq8}
 -\frac{F^2}{4G} = -\frac{1}{4}\left(\frac{\mu^2(k - j)^2)}{4\tau^2}\right) \frac{4\tau}{k^2 - j^2} = -\frac{\mu^2(k - j)}{4\tau(k+j)}.
 \end{align}
Note that when $\mu \neq 0$, then the only difference between the matrix $\frac{1}{S_{y^2}}N_{\tau, \mu}$ and the matrix $W_{\beta, \gamma}$ in (\ref{Eq7}), is the constant in the argument of the trace function. In particular, if $N = N_{\tau, \mu}$, then 
\[
\frac{1}{S_{y^2}}N_{j, k} =  \left\{
 \begin{array}{ll}
  0, & \mbox{\text{if }$j^2 = k^2$}, \\
\eta(\tau(k^2 - j^2))\zeta^{-tr(\frac{\mu^2(k - j)}{4\tau(k + j)})}, & \mbox{\text{otherwise}}. 
\end{array}\right.
\]

As a result, the following lemma is immediate.
 
\begin{lem}
Let $q = p^e$ be an odd prime power, $\tau, \mu \in \F_q^*$. Then the eigenvalues of $\frac{1}{S_{y^2}}N_{\tau, \mu}(S)$ are:
\begin{enumerate}
    \item For each non-trivial $\chi \in \widehat{\F_q^*}$,   $$
    \lambda_\chi =  \eta(\tau)\left( \sum_{\substack{t\in \mathbb{F}_q\\ t\neq -1}} \chi(t)\eta(t^2- 1)\zeta^{-tr(\frac{\mu^2(t - 1)}{4\tau(t + 1)})}\right).
    $$
    \item  $\lambda_1, \lambda_2$, where $\lambda_1, \lambda_2$ are the roots of 
      \begin{enumerate}[(i)]
    \item
    $
\lambda^2 - (S_{y^2} -2\cos(\frac{2\pi}{p}tr(\frac{\mu^2}{4\tau})))\lambda  - (q - 1) = 0
$
if $q \equiv 1$(mod 4).
\item 
    $
\lambda^2 - (S_{y^2} - 2\cos(\frac{2\pi}{p}tr(\frac{\mu^2}{4\tau})))\lambda  + (q - 1) = 0
$
if $q \equiv 3$(mod 4).
\end{enumerate}
\end{enumerate}
Furthermore, each eigenvalue $\lambda$ satisfies $|\lambda| \leq 3\sqrt{q}$.
\end{lem}

Thus, if $\lambda$ is an eigenvalue of $N_{\tau, \mu}$, then $|\lambda| \leq 3q$. Hence, we prove Theorem \ref{Main_Theorem}.

\section{Final Remarks}\label{Section_Finalremarks}
In this paper, we prove that the second largest eigenvalue of $D(5,q)$ is less than or equal to $2\sqrt{q}$ by instead studying the spectrum of their point collinearity graph. Our proof assumes that $q$ is an odd prime power. If $q$ is an even prime power we managed to find the Cayley graph description, but could not compute its irreducible representations yet. We choose to omit the details in this paper. The conjecture \ref{Main_Conjecture} is still wide open for $k\geq 7$. 
We also note that the covering property of the family $D(k,q)$ ensures that the spectrum of $D(k,q)$ embeds in that of $D(k+1,q)$, see, e.g., \cite[Sec.\ 3C]{LW}.

\section{Acknowledgements}
The research of V.\ Taranchuk was  partially supported by the Simons Foundation Award ID: 426092 and National Sciency Foundation Grant: 1855723.
The authors would like to thank Sebastian Cioab\u{a} for suggesting this interesting problem and Felix Lazebnik for his useful comments that improved the exposition of this paper. We would also like to thank Jason Williford for sharing his insights which helped the authors find the irreducible representations of the group $G$. 


\end{document}